\documentclass
{amsart}
\usepackage{graphicx}
\newtheorem{thm}{Theorem}[section]
\theoremstyle{definition}
\newtheorem{cor}[thm]{Corollary}
\newtheorem{prop}[thm]{Proposition}
\newtheorem{defn}[thm]{Definition}
\newtheorem{lem}[thm]{Lemma}

\newtheorem{rem}[thm]{Remark}
\newtheorem{ex}[thm]{Example}

\numberwithin{equation}{section}
\begin{document}
\title[2-absorbing and strongly 2-absorbing secondary submodules of modules]
{2-absorbing and strongly 2-absorbing secondary submodules of modules \footnote{This article was accepted for publication in Le Matematiche.}}

\author{H. Ansari-Toroghy}
\address {Department of pure Mathematics, Faculty of mathematical
Sciences, University of Guilan,
P. O. Box 41335-19141, Rasht, Iran.}%
\email{ansari@guilan.ac.ir}%

\author{F. Farshadifar}
\footnote {This research was in part supported by a grant from IPM (No. 94130048)}
\address{University of Farhangian, P. O. Box 19396-14464, Tehran, Iran.}
\email{f.farshadifar@gmail.com}

\address{School of Mathematics, Institute for Research in Fundamental Sciences (IPM), P.O. Box: 19395-5746, Tehran, Iran.}

\begin{abstract}
In this paper, we will introduce the concept of 2-absorbing (resp. strongly 2-absorbing) secondary submodules of modules over a commutative ring as a generalization of secondary modules and investigate some basic properties of these classes of modules.
\end{abstract}
\maketitle
\section{Introduction}
\noindent

Throughout this paper, $R$ will denote a commutative ring with
identity and $\Bbb Z$ will denote the ring of integers.

Let $M$ be an $R$-module. A proper submodule $P$ of $M$ is said
 to be \emph{prime} if for any $r \in R$ and $m \in M$ with
$rm \in P$, we have $m \in P$ or $r \in (P:_RM)$ \cite{Da78}.
Let $N$ be a proper submodule of $M$. Then the \emph{ $M$-radical} of $N$, denoted by $M$-$rad(N)$, is defined to be the
intersection of all prime submodules of $M$ containing $N$. If $M$ has no prime
submodule containing $N$, then the $M$-radical of $N$ is defined to be $M$ \cite{MM86}. A non-zero submodule $S$ of $M$ is said to be \emph{second}
if for each $a \in R$, the homomorphism $ S \stackrel {a} \rightarrow S$
is either surjective or zero \cite{Y01}.
In this case $Ann_R(S)$ is a prime ideal of $R$.

The notion of 2- absorbing ideals as a generalization of prime ideals was introduced and studied  in \cite{Ba07}. A proper ideal $I$ of $R$ is a \emph{2-absorbing ideal}
of $R$ if whenever $a, b, c \in R$ and $abc \in I$, then $ab \in I$ or
 $ac \in I$ or $bc \in I$. It has been proved that $I$ is a 2-absorbing ideal of
 $R$ if and only if whenever $I_1$, $I_2$, and $I_3$ are ideals of $R$
with $I_1I_2I_3 \subseteq I$,
then $I_1I_2 \subseteq I$ or $I_1I_3 \subseteq I$ or $I_2I_3 \subseteq I$ \cite{Ba07}.
In \cite{BUY14}, the authors introduced the concept of 2-absorbing primary ideal which is a generalization of primary ideal. A proper ideal $I$ of $R$ is called a \emph{2-absorbing primary ideal} of $R$ if whenever $a, b, c \in R$ and $abc \in I$, then $ab\in I$ or $ac\in \sqrt{I}$ or  $bc\in \sqrt{I}$.


The notion of 2-absorbing ideals was extended to 2-absorbing submodules in \cite{YS11}.
A proper submodule $N$ of $M$ is called a \emph{2-absorbing submodule }of
 $M$ if whenever $abm \in N$
for some $a, b \in R$ and $m \in M$, then $am \in N$ or $bm \in N$ or
$ab \in (N :_R M)$.

 In \cite{AF16}, the  present authors introduced the dual notion of 2-absorbing submodules (that is, \emph{2-absorbing (resp. strongly 2-absorbing) second submodules}) of $M$ and investigated some properties of these classes of modules. A non-zero submodule $N$ of $M$ is said to be a \emph{2-absorbing second submodule of} $M$ if whenever
 $a, b \in R$, $L$ is a completely irreducible submodule of $M$,
and $abN\subseteq L$, then $aN\subseteq L$ or $bN\subseteq L$ or $ab \in Ann_R(N)$.
A non-zero submodule $N$ of $M$ is said to be a \emph{strongly 2-absorbing second submodule of} $M$ if whenever
 $a, b \in R$, $K$ is a submodule of $M$,
and $abN\subseteq K$, then $aN\subseteq K$ or $bN\subseteq K$ or $ab \in Ann_R(N)$.

In \cite{mtoa16}, the authors introduced the notion of 2-absorbing
primary submodules as a generalization of 2-absorbing primary ideals of rings and studied some properties of this class of modules.
A proper submodule $N$ of $M$ is said to be a \emph{2-absorbing primary submodule} of $M$ if whenever $a, b \in R$, $m \in M$, and $abm \in N$, then $am \in M$-$rad(N)$ or $bm \in M$-$rad(N)$  or $ab \in (N :_R M)$.

The purpose of this paper is to introduce the concepts of 2-absorbing and strongly 2-absorbing secondary submodules of an $R$-module $M$ as dual notion of 2-absorbing primary submodules and obtain some related results.

\section{Main results}
Let $M$ be an $R$-module. For a submodule $N$ of $M$ the the \emph{second radical} (or \emph{\emph{second socle}}) of $N$ is defined  as the sum of all second submodules of $M$ contained in $N$ and it is denoted by $sec(N)$ (or $soc(N)$). In case $N$ does not contain any second submodule, the second radical of $N$ is defined to be $(0)$. $N \not =0$ is said to be a \emph{second radical submodule of $M$} if $sec(N)=N$ (see \cite{CAS13} and \cite{AF11}).

 A proper submodule $N$ of $M$ is said to be \emph{completely irreducible} if $N=\bigcap _
{i \in I}N_i$, where $ \{ N_i \}_{i \in I}$ is a family of
submodules of $M$, implies that $N=N_i$ for some $i \in I$. It is
easy to see that every submodule of $M$ is an intersection of
completely irreducible submodules of $M$ \cite{FHo06}.

We frequently use the following basic fact without further comment.
\begin{rem}\label{r2.1}
Let $N$ and $K$ be two submodules of an $R$-module $M$. To prove $N\subseteq K$, it is enough to show that if $L$ is a completely irreducible submodule of $M$ such that $K\subseteq L$, then $N\subseteq L$.
\end{rem}

\begin{defn}\label{d9.1}
We say that a non-zero submodule $N$ of an $R$-module $M$ is a
\emph{2-absorbing secondary submodule} of $M$ if whenever $a, b \in R$, $L$ is a completely irreducible submodule of $M$ and $abN\subseteq L$,
then $a(sec(N)) \subseteq L$ or $b(sec(N)) \subseteq L$ or $ab \in Ann_R(N)$.
By a \emph{2-absorbing secondary module}, we mean a module which is a 2-absorbing secondary submodule of itself.
\end{defn}

\begin{ex}\label{e9.2}
Clearly, every submodule of the $\Bbb Z$-module $\Bbb Z$ is not secondary. But as $sec(\Bbb Z)=0$,  every submodule of the $\Bbb Z$-module $\Bbb Z$ is 2-absorbing secondary.
\end{ex}

\begin{lem}\label{l1.3}
Let $I$ be an ideal of $R$ and $N$ be a 2-absorbing secondary submodule of $M$.
If $a\in R$, $L$ is a completely irreducible submodule of $M$,
and $IaN \subseteq L$, then $a(sec(N)) \subseteq L$ or $I(sec(N)) \subseteq L$ or $Ia \in Ann_R(N)$.
\end{lem}
\begin{proof}
Let $a(sec(N)) \not \subseteq L$ and $Ia \not \in Ann_R(N)$.
Then there exists $b \in I$ such that $abN \not = 0$. Now as $N$
is a 2-absorbing secondary submodule of $M$, $baN \subseteq L$ implies that $b(sec(N)) \subseteq L$. We show that $I(sec(N)) \subseteq L$. To see this, let $c$ be an arbitrary element of $I$.
Then $(b + c)aN \subseteq L$. Hence, either $(b + c)(sec(N)) \subseteq L$ or $(b + c)a \in Ann_R(N)$. If $(b+c)(sec(N)) \subseteq L$, then since $b(sec(N)) \subseteq L$ we have $c(sec(N)) \subseteq L$. If $(b+c)a \in Ann_R(N)$, then $ca \not \in Ann_R(N)$. Thus $caN \subseteq L$ implies that $c(sec(N)) \subseteq L$. Hence, we conclude that
$I(sec(N) \subseteq L$.
\end{proof}

\begin{thm}\label{l1.4}
Let $I$ and $J$ be two ideals of $R$ and $N$ be a 2-absorbing secondary submodule of an $R$-module  $M$.
If $L$ is a completely irreducible submodule of $M$ and $IJN \subseteq L$,
then $I(sec(N)) \subseteq L$ or $J(sec(N)) \subseteq L$ or $IJ \subseteq Ann_R(N)$.
\end{thm}
\begin{proof}
Let $I(sec(N)) \not \subseteq L$ and $J(sec(N)) \not \subseteq L$. We show that $IJ \subseteq Ann_R(N)$. Assume that $c \in I$ and $d\in J$. By assumption, there exists $a\in I$ such that
$a(sec(N)) \not \subseteq L$ but $aJN \subseteq L$. Now Lemma \ref{l1.3}
shows that $aJ \subseteq Ann_R(N)$ and so $(I\setminus (L:_Rsec(N)))J\subseteq Ann_R(N)$.
Similarly, there exists $b \in (J\setminus (L:_Rsec(N)))$ such that $Ib \subseteq Ann_R(N)$ and also $I(J\setminus (L:_Rsec(N)))\subseteq Ann_R(N)$. Thus we have $ab \in Ann_R(N)$,
$ad \in Ann_R(N)$, and $cb \in Ann_R(N)$. As $a + c \in I$
and $b + d \in J$, we have $(a + c)(b + d)N \subseteq L$.
Therefore, $(a + c)(sec(N)) \subseteq L$ or $(b + d)(sec(N)) \subseteq L$ or $(a + c)(b + d) \in Ann_R(N)$. If $(a + c)(sec(N)) \subseteq L$, then $c(sec(N)) \not \subseteq L$. Hence $c \in  I\setminus (L:_Rsec(N))$
which implies that $cd \in Ann_R(N)$. Similarly, if $(b + d)(sec(N)) \subseteq L$, we can deduce that $cd \in Ann_R(N)$. Finally if $(a+c)(b+d) \in Ann_R(N)$, then $ab+ad+cb+cd \in Ann_R(N)$ so that $cd \in Ann_R(N)$. Therefore, $IJ \subseteq Ann_R(N)$.
\end{proof}

\begin{thm}\label{t1.5}
Let $N$ be a non-zero submodule of an $R$-module $M$. The following statements
are equivalent:
\begin{itemize}
 \item [(a)]  If $abN\subseteq L_1 \cap L_2$ for some $a, b \in R$ and completely irreducible submodules $L_1, L_2$ of $M$, then $a(sec(N)) \subseteq L_1 \cap L_2$ or $b(sec(N)) \subseteq L_1 \cap L_2$ or $ab \in Ann_R(N)$;
 \item [(b)] If $IJN \subseteq K$ for some ideals $I, J$ of $R$ and a
submodule $K$ of $M$, then $I(sec(N)) \subseteq K$ or $J(sec(N)) \subseteq K$ or $IJ \in Ann_R(N)$;
\item [(c)] For each $a , b \in R$, we have $a(sec(N)) \subseteq abN$ or  $b(sec(N)) \subseteq abN$ or $abN=0$.
\end{itemize}
\end{thm}
\begin{proof}
$(a) \Rightarrow (b)$. Assume that $IJN \subseteq K$ for some ideals $I, J$ of $R$, a submodule $K$ of $M$, and $IJ \not \subseteq Ann_R(N)$.
Then by Theorem \ref{l1.4}, for all completely irreducible submodules $L$ of $M$ with $K \subseteq L$ either $I(sec(N)) \subseteq L$ or $J(sec(N))\subseteq L$. If $I(sec(N)) \subseteq L$ (resp. $J(sec(N)) \subseteq L$) for all completely irreducible submodules $L$ of $M$ with $K \subseteq L$, we are done. Now suppose that $L_1$ and $L_2$
are two completely irreducible submodules of $M$ with $K \subseteq L_1$, $K \subseteq L_2$, $I(sec(N)) \not \subseteq L_1$, and $J(sec(N)) \not \subseteq L_2$.
Then $I(sec(N))  \subseteq L_2$
and $J(sec(N))  \subseteq L_1$. Since $IJN \subseteq L_1 \cap L_2$, we
have either $I(sec(N)) \subseteq L_1 \cap L_2$ or $J(sec(N)) \subseteq L_1 \cap L_2$.
If $I(sec(N)) \subseteq L_1 \cap L_2$, then $I(sec(N)) \subseteq L_1$ which is a
contradiction. Similarly from $J(sec(N)) \subseteq L_1 \cap L_2$ we get a contradiction.

$(b) \Rightarrow (a)$. This is clear

$(a) \Rightarrow (c)$. By part (a), $N \not =0$. Let $a, b \in R$. Then $abN \subseteq abN$ implies that $a(sec(N)) \subseteq abN$ or $b(sec(N)) \subseteq abN$ or $abN=0$.

$(c) \Rightarrow (a)$. This is clear.
\end{proof}

\begin{defn}\label{d9.1}
We say that a non-zero submodule $N$ of an $R$-module $M$ is a
\emph{strongly 2-absorbing secondary submodule} of $M$ if
satisfies the equivalent conditions of Theorem \ref{t1.5}.
By a \emph{strongly 2-absorbing secondary module}, we mean a module which is a strongly 2-absorbing secondary submodule of itself.
\end{defn}

Let $N$ be a submodule of an $R$-module $M$. Then part (d) of Theorem \ref{t1.5}   shows that $N$ is a strongly 2-absorbing secondary submodule of $M$ if and only if $N$ is a strongly 2-absorbing secondary module.

\begin{ex}\label{e9.1}
Clearly every strongly 2-absorbing secondary submodule is a 2-absorbing secondary submodule. But the converse is not true in general. For example, consider $M=\Bbb Z_6 \oplus \Bbb Q$ as a $\Bbb Z$-module. Then $M$ is a
2-absorbing secondary module. But since $0\not= 6M \subseteq 0 \oplus \Bbb Q$,  $sec(M)=M$, $2M\not \subseteq 0 \oplus \Bbb Q$, and $3M \not \subseteq 0 \oplus \Bbb Q$, $M$ is not a strongly 2-absorbing secondary module.
\end{ex}

\begin{prop}\label{l9.2}
Let $N$  be a 2-absorbing second submodule of an $R$-module $M$. Then $N$ is a strongly 2-absorbing secondary submodule  of $M$.
\end{prop}
\begin{proof}
Let $a, b \in R$ and  $K$ be a submodule of $M$ such that $abN \subseteq K$. Then $aN \subseteq K$ or $bN \subseteq K$ or $ab N=0$ by assumption.
Thus $a(sec(N))\subseteq aN \subseteq K$ or $b(sec(N)) \subseteq aN \subseteq K$ or $abN=0$, as required.
\end{proof}

The following example shows that the converse of the Proposition \ref{l9.2} is not true in general.
\begin{ex}\label{ee9.3}
Let $M$ be the $\Bbb Z$-module $\Bbb Z_{p^\infty}$. Then as $p^2 \langle 1/p^3+\Bbb Z \rangle \subseteq \langle 1/p+\Bbb Z \rangle$,  $p \langle 1/p^3+\Bbb Z \rangle \not \subseteq \langle 1/p+\Bbb Z \rangle$, and $p^2 \langle 1/p^3+\Bbb Z \rangle \not =0$, we have the submodule  $\langle 1/p^3+\Bbb Z \rangle$ of $\Bbb Z_{p^\infty}$ is not 2-absorbing second submodule. But $sec(\langle 1/p^3+\Bbb Z \rangle)=\langle 1/p+\Bbb Z \rangle$ implies that  $\langle 1/p^3+\Bbb Z \rangle$ is a strongly 2-absorbing secondary submodule of $M$.
\end{ex}

An $R$-module $M$ is said to be a \emph{comultiplication module} if for every submodule $N$ of $M$ there exists an ideal $I$ of $R$ such that $N=(0:_MI)$, equivalently, for each submodule $N$ of $M$, we have $N=(0:_MAnn_R(N))$ \cite{AF07}.
\begin{thm}\label{t9.4}
Let $M$ be a finitely generated comultiplication $R$-module. If $N$
is a strongly 2-absorbing secondary submodule of $M$, then $Ann_R(N)$ is a 2-absorbing primary ideal of $R$.
\end{thm}
\begin{proof}
Let $a, b, c \in R$ be such that $abc \in Ann_R(N)$, $ac \not \in  \sqrt{Ann_R(N)}$, and $bc \not \in  \sqrt{Ann_R(N)}$. Since by \cite[2.12]{AF25}, $Ann_R(sec(N))=\sqrt{Ann_R(N)}$, there exist completely irreducible submodules $L_1$ and $L_2$ of $M$ such that $ac(sec(N)) \not \subseteq L_1$ and $bc(sec(N)) \not \subseteq L_2$. But $abcN=0 \subseteq L_1 \cap L_2$ implies that $abN \subseteq (L_1 \cap L_2:_Mc)$. Now as $N$ is a strongly 2-absorbing secondary submodule of $M$, we have $a(sec(N)) \subseteq (L_1 \cap L_2:_Mc)$ or $b(sec(N)) \subseteq (L_1 \cap L_2:_Mc)$  or $abN=0$. If $a(sec(N)) \subseteq (L_1 \cap L_2:_Mc)$ (resp. $b(sec(N)) \subseteq (L_1 \cap L_2:_Mc)$), then $ac(sec(N)) \subseteq L_1$ (resp. $bc(sec(N))  \subseteq L_2$), a contradiction. Hence $abN=0$, as needed.
\end{proof}

\begin{thm}\label{t9.5}
Let $N$ be a submodule of a comultiplication $R$-module $M$. If $Ann_R(N)$ is a 2-absorbing primary ideal of $R$, then $N$ is a strongly 2-absorbing secondary submodule of $M$.
\end{thm}
\begin{proof}
Let $abN \subseteq K$ for some $a, b \in R$ and some submodule $K$ of
$M$. As $M$ is a comultiplication module, there exists an ideal $I$ of $R$ such that $K =(0:_MI)$. Hence $Iab \subseteq Ann_R(N)$ which implies that either $Ia \subseteq \sqrt{Ann_R(N)}$ or $Ib \subseteq \sqrt{Ann_R(N)}$ or $ab \in Ann_R(N)$. If $ab \in Ann_R(N)$, we are done. If $Ia \subseteq \sqrt{Ann_R(N)}$, as  $\sqrt{Ann_R(N)}\subseteq Ann_R(sec(N))$, we have $Ia(sec(N))=0$. This implies that $a(sec(N)) \subseteq K$ because $M$ is a comultiplication module. Similarly, if $Ib \subseteq \sqrt{Ann_R(N)}$, we get $b(sec(N)) \subseteq K$. This completes the proof.
\end{proof}

The following example shows that Theorem \ref{t9.5} is not satisfied in general.
\begin{ex}\label{e9.5}
Consider the $\Bbb Z$-module $M = \Bbb Z_p\oplus \Bbb Z_q\oplus \Bbb Q$, where $p \not =q$ are two prime numbers. Then $M$ is not a comultiplication $\Bbb Z$-module and $Ann_{\Bbb Z}(M)=0$ is a 2-absorbing primary ideal of $R$. But since $0\not= pqM \subseteq 0 \oplus 0 \oplus \Bbb Q$, $sec(M)=M$, $pM\not \subseteq 0 \oplus 0 \oplus \Bbb Q$, and $qM \not \subseteq 0 \oplus 0 \oplus \Bbb Q$, $M$ is not a strongly 2-absorbing secondary module.
\end{ex}

In \cite[2.6]{mtoa16}, it is shown that, if $M$ is a finitely generated multiplication $R$-module and $N$ is a 2-absorbing primary submodule of $M$, then
$M$-$rad(N)$ is a 2-absorbing submodule of $M$. In the following lemma, we see that some of this conditions are redundant.

\begin{lem}\label{l9.6}
Let $N$ be a 2-absorbing primary submodule of an $R$-module $M$. Then
$M$-$rad(N)$ is a 2-absorbing submodule of $M$.
\end{lem}
\begin{proof}
This follows from the fact that $M$-$rad(M$-$rad(N))=M$-$rad(N)$ by \cite[Proposition 2]{Lu89}.
\end{proof}

\begin{prop}\label{p1.5}
Let $M$ be an $R$-module. Then we have the following.
\begin{itemize}
  \item [(a)]
If $N$ is a 2-absorbing (resp. strongly 2-absorbing) secondary submodule of an $R$-module $M$, then $sec(N)$ is a 2-absorbing (resp. strongly 2-absorbing) second submodule of $M$.
  \item [(b)]
If $N$ is a second radical submodule of $M$, then $N$ is a 2-absorbing (resp. strongly 2-absorbing) second submodule if and only if $N$ is a 2-absorbing (resp. strongly 2-absorbing) secondary submodule.
\end{itemize}
\end{prop}
\begin{proof}
(a) This follows from the fact that $sec(sec(N))=sec(N)$ by \cite[2.1]{AF25}.

(b) This follows from part (a)
\end{proof}

Let $N$ and $K$ be two submodules of an $R$-module $M$. The \emph{coproduct} of $N$ and $K$ is defined by $(0:_MAnn_R(N)Ann_R(K))$ and denoted by $C(NK)$ \cite{AF007}.

\begin{thm}\label{t9.7}
Let $N$ be a submodule of an $R$-module $M$ such that $sec(N)$ is a second submodule of $M$. Then we have the following.
\begin{itemize}
  \item [(a)] $N$ is a strongly 2-absorbing secondary submodule of $M$.
  \item [(b)] If $M$ is a comultiplication $R$-module, then $C(N^t)$ is a strongly 2-absorbing secondary submodule of $M$ for every positive integer $t\geq 1$, where $C(N^t)$ means the coproduct of N, t times.
\end{itemize}
\end{thm}
\begin{proof}
(a) Let $a, b \in R$, $K$ be a submodule of $M$ such that $abN \subseteq K$, and $b(sec(N)) \not \subseteq K$. Then as $sec(N)$ is a second submodule and $a(sec(N)) \subseteq aN \subseteq (K:_Mb)$, we have $a(sec(N))=0$ by \cite[2.10]{AF12}. Thus $a(sec(N)) \subseteq K$, as needed.

(b) Let $M$ be a comultiplication $R$-module. Then there exists an ideal $I$ of $R$ such that $N = (0:_MI)$. Thus by \cite[2.1]{AF25},
$$
sec(c(N^t))=sec((0:_MI^t))=sec((0:_MI))=sec(N).
$$
Now the results follows from to the proof of part (a).
\end{proof}

\begin{thm}\label{t9.8}
Let $M$ be a comultiplication $R$-module.
Then we have the following.
\begin{itemize}
  \item [(a)] If $N_1, N_2, ..., N_n$ are strongly
2-absorbing secondary submodules of $M$ with the same second radical, then $N =\sum_{i=1}^n N_i$  is a  strongly 2-absorbing secondary submodule of $M$.
  \item [(b)] If $N_1, N_2, ..., N_n$ are 2-absorbing secondary submodules of $M$ with the same second radical, then $N =\sum_{i=1}^n N_i$  is a  2-absorbing secondary submodule of $M$.
  \item [(c)] If $N_1$ and $N_2$ are two secondary submodules of $M$, then $N_1+N_2$ is a strongly 2-absorbing secondary submodule of $M$.
  \item [(d)] If $M$ is finitely generated, $N$ is a submodule
of $M$ which possess a secondary representation, and $sec(N)=K_1+K_2$, where $K_1$ and $K_2$ are two minimal submodules of $M$, then $N$ is a strongly 2-absorbing secondary submodule of $M$.
\end{itemize}
\end{thm}
\begin{proof}
(a) Let $a, b \in R$ and $K$ be a submodule of $M$ such that $abN \subseteq K$. Thus for each $i=1,2, ...,n$, $abN_i \subseteq K$. If there exists $1 \leq j \leq n$ such that $a(sec(N_j)) \subseteq K$ or $b(sec(N_j)) \subseteq K$, then
$a(sec(N)) \subseteq K$  or $b(sec(N)) \subseteq K$ (note that $sec(N)=sec(\sum_{i=1}^nN_i)=\sum_{i=1}^nsec(N_i)=sec(N_i)$ by \cite[2.6]{CAS13}).  Otherwise, $abN_i=0$ for each $i=1,2,...,n$. Hence $abN=0$, as desired.

(b) The proof is similar to the part (a).

(c) As $N_1$ and $N_2$ are secondary submodules of $M$,  $Ann_R(N_1)$ and $Ann_R(N_2)$ are primary ideals of $R$. Hence $Ann_R(N_1+N_2)=Ann_R(N_1) \cap Ann_R(N_2)$ is a 2-absorbing primary ideal of $R$ by \cite[2.4]{BUY14}. Thus by Theorem \ref{t9.5}, $N_1+ N_2$ is a strongly 2-absorbing secondary submodule of $M$.

(d) Let $N = \sum_{i=1}^nN_i$  be a secondary representation. By \cite[2.6]{AF25},
 $sec(N)=\sum_{i=1}^nsec(N_i)$. Since $sec(N_i)$'s are second submodules of $M$ by \cite[2.13]{AF25}, we have
 $$
 \{sec(N_1),sec(N_2), ..., sec(N_n)\}=\{K_1, K_2\}.
 $$
Without loss of generality, we
may assume that for some $1 \leq  t < n$, $\{sec(N_1), ..., sec(N_t)\}=\{K_1\}$ and $\{sec(N_{t+1}), ..., sec(N_n)\}=\{K_2\}$. Set $H_1:= N_1+...+N_t$ and $H_2:=N_{t+1}+...N_n$. By \cite[2.12]{AF25}, $H_1$ and $H_2$ are secondary submodules of $M$. Therefore, by part (c),
$N =H_1+H_2$ is a strongly 2-absorbing secondary submodule of $M$.
\end{proof}

The following example shows that the direct sum of two strongly 2-absorbing secondary $R$-modules is not a strongly 2-absorbing secondary $R$-module in general.
\begin{ex}\label{t9.7}
Clearly, the $\Bbb Z$-modules $\Bbb Z_6$ and $\Bbb Z_{10}$ are strongly 2-absorbing secondary $\Bbb Z$-modules. Let $M=\Bbb Z_6\oplus \Bbb Z_{10}$. Then $M$ is not a strongly 2-absorbing second $\Bbb Z$-module. By \cite[2.1]{AF12}, $sec(M)=M$. Thus $M$ is not a strongly 2-absorbing secondary $\Bbb Z$-module by Proposition  \ref{p1.5}.
\end{ex}

\begin{lem}\label{l9.9}
Let $f : M \rightarrow \acute{M}$ be a monomorphism of R-modules. Then we have the following.
\begin{itemize}
  \item [(a)] If $N$ is a submodule of $M$, then $sec(f(N))= f(sec(N))$.
  \item [(b)] If $\acute{N}$ is a submodule of $\acute{M}$ such that $\acute{N} \subseteq f(M)$, then $sec(f^{-1}(\acute{N}))=f^{-1}(sec(\acute{N}))$.
\end{itemize}
\end{lem}
\begin{proof}
(a) Let $\acute{S}$ be a second submodule of $f(N)$. Then one can see that $f^{-1}(\acute{S})$ is a second submodule of $N$. Thus $f(f^{-1}(\acute{S})) \subseteq f(sec(N))$. Therefore, $sec(f(N))\subseteq f(sec(N))$. The reverse inclusion is clear.

(b) Let $S$ be a second submodule of $f^{-1}(\acute{N})$. Then one can see that $f(S)$ is a second submodule of $\acute{N}$. Thus $f^{-1}(S) \subseteq f^{-1}(sec(\acute{N}))$. Therefore, $sec(f^{-1}(\acute{N}))\subseteq f^{-1}(sec(\acute{N}))$. To see the reverse inclusion, let $\acute{S}$ be a second submodule of $\acute{N}$. Then $f^{-1}(\acute{S})$ is a second submodule of $f^{-1}(\acute{N})$. This implies that $f^{-1}(sec(\acute{N})) \subseteq sec(f^{-1}(\acute{N}))$.
\end{proof}

\begin{thm}\label{t9.10}
Let $f : M \rightarrow \acute{M}$ be a monomorphism of R-modules. Then we have the following.
\begin{itemize}
  \item [(a)] If $N$ is a strongly 2-absorbing secondary submodule of $M$, then $f(N)$ is a strongly 2-absorbing secondary submodule of $\acute{M}$.
  \item [(b)] If $\acute{N}$ is a strongly 2-absorbing secondary submodule of $\acute{M}$ and $\acute{N} \subseteq f(M)$, then $f^{-1}(\acute{N})$ is a strongly 2-absorbing secondary submodule of $M$.
 \end{itemize}
\end{thm}
\begin{proof}
(a) Since $N \not =0$ and $f$ is a monomorphism, we have $f(N) \not =0$. Let $a, b \in R$, $\acute{K}$ be a submodule of $\acute{M}$, and $abf(N)\subseteq \acute{K}$. Then $abN \subseteq f^{-1}(\acute{K})$. As $N$ is strongly 2-absorbing secondary submodule, $a(sec(N)) \subseteq f^{-1}(\acute{K})$ or $b(secN)) \subseteq f^{-1}(\acute{K})$ or $abN=0$. Therefore, by Lemma \ref{l9.9} (a),
 $$
 a(sec(f(N)))= a(f(sec(N))) \subseteq f(f^{-1}(\acute{K}))=f(M) \cap \acute{K} \subseteq \acute{K}
 $$
 or
 $$
 b(sec(f(N)))= b(f(sec(N))) \subseteq f(f^{-1}(\acute{K}))=f(M) \cap \acute{K} \subseteq \acute{K}
 $$
 or $abf(N)=0$, as needed.

(b) If $f^{-1}(\acute{N})=0$, then $f(M) \cap \acute{N}=f(f^{-1}(\acute{N}))=f(0)=0$. Thus $\acute{N}=0$, a contradiction. Therefore, $f^{-1}(\acute{N})\not=0$. Now let $a, b \in R$, $K$ be a submodule of $M$, and $abf^{-1}(\acute{N})\subseteq K$. Then
  $$
  ab\acute{N}=ab(f(M) \cap \acute{N})=abff^{-1}(\acute{N})\subseteq f(K).
  $$
 As $\acute{N}$ is strongly 2-absorbing secondary submodule, $a(sec(\acute{N}) \subseteq f(K)$ or $b(sec(\acute{N}) \subseteq f(K)$ or $ab\acute{N}=0$. Hence by Lemma \ref{l9.9} (b), $a(sec(f^{-1}(\acute{N})))=af^{-1}(sec(\acute{N}))\subseteq f^{-1}(f(K))=K$ or $ b(sec(f^{-1}(\acute{N})))=bf^{-1}(sec(\acute{N}))\subseteq f^{-1}(f(K))=K$ or $abf^{-1}(\acute{N})=0$, as desired.
\end{proof}

\begin{cor}\label{c9.11}
Let $M$ be an $R$-module and let $N\subseteq K$ be two submodules of $M$. Then $N$ is a strongly 2-absorbing secondary submodule of $K$ if and only if  $N$ is a strongly 2-absorbing secondary submodule of $M$.
\end{cor}
\begin{proof}
This follows from Theorem \ref{t9.10} by using the natural monomorphism $K\rightarrow M$.
\end{proof}

\begin{prop}\label{p9.12}
Let $M$ be a cocyclic $R$-module with minimal submodule $K$ and $N$ be a submodule of $M$ such that $rN \not =K$ for each $r \in R$. If $N/K$ is a strongly 2-absorbing secondary submodule of $M/K$, then $N$ is a strongly 2-absorbing secondary submodule of $M$.
\end{prop}
\begin{proof}
Let $a, b \in R$ and $H$ be a submodule of $M$ such that $abN \subseteq H$. Then $ab(N/K)\subseteq H/K$ implies that $a(sec(N/K))\subseteq H/K$ or $b(sec(N/K))\subseteq H/K$ or $ab(N/K)=0$. If $ab(N/K)=0$, then $abN=0$ because
 $rN \not =K$ for each $r \in R$. Otherwise, since $a(sec(N))/K \subseteq sec(N/K)$, we have $a(sec(N))\subseteq H$ or $b(sec(N))\subseteq H$ as required.
\end{proof}

Let $R_i$ be a commutative ring with identity and $M_i$ be an $R_i$-module, for $i = 1, 2$. Let $R = R_1 \times R_2$. Then $M = M_1 \times M_2$ is an $R$-module and each submodule of $M$ is in the form of $N = N_1 \times N_2$ for some submodules $N_1$ of $M_1$ and $N_2$ of $M_2$. In addition, $M_i$ is a comultiplication $R_i$-module, for $i = 1, 2$ if and only if $M$ is a comultiplication $R$-module by \cite[2.1]{NS14}.

\begin{lem}\label{l9.13}
Let $R = R_1 \times R_2$ and $M = M_1 \times M_2$, where $M_1$ is an $R_1$-module and $M_2$ is an $R_2$-module. If $N=N_1 \times N_2$ is a submodule of $M$, then we have the following.
\begin{itemize}
  \item [(a)] $N$ is a second submodule of $M$ if and only if $N=S_1\times 0$ or $N=S_2 \times 0$, where $S_1$ is a second submodule of $ N_1$ and $S_2$ is a second submodule of $M_2$.
  \item [(b)] $sec(N)=sec(N_1) \times sec(N_2)$.
\end{itemize}
\end{lem}
\begin{proof}
(a) This is straightforward.

(b) This follows from part (a).
\end{proof}
\begin{thm}\label{t9.14}
Let $R = R_1 \times R_2$ and $M = M_1 \times M_2$, where $M_1$ is a
comultiplication $R_1$-module and $M_2$ is a comultiplication $R_2$-module.
Then we have the following.
\begin{itemize}
  \item [(a)] If $M_1$ be a finitely generated $R_1$-module, then a non-zero submodule $K_1$ of $M_1$ is a strongly 2-absorbing secondary submodule if and only if $N = K_1 \times 0$ is a strongly 2-absorbing secondary submodule of $M$.
  \item [(b)] If $M_2$ be a finitely generated $R_2$-module, then a non-zero submodule $K_2$ of $M_2$ is a strongly 2-absorbing secondary submodule if and only if $N =0 \times K_2$ is a strongly 2-absorbing secondary submodule of $M$.
  \item [(c)] If $K_1$ is a secondary submodule of $M_1$ and $K_2$ is a secondary submodule of $M_2$, then $N = K_1 \times K_2$ is a strongly 2-absorbing secondary submodule of $M$.
\end{itemize}
\end{thm}
\begin{proof}
(a) Let $K_1$ be a strongly 2-absorbing secondary submodule of $M_1$. Then $Ann_{R_1}(K_1)$ is a 2-absorbing primary ideal of $R_1$ by Theorem \ref{t9.4}. Now since $Ann_R(N)=Ann_{R_1}(K_1) \times R_2$, we have $Ann_R(N)$ is a 2-absorbing primary ideal of $R$ by \cite[2.23]{BUY14}. Thus the result follows from Theorem \ref{t9.5}. Conversely, let $N = K_1 \times 0$ be a strongly 2-absorbing secondary submodule of $M$. Then $Ann_R(N)=Ann_{R_1}(K_1) \times R_2$ is a primary ideal of $R$ by Theorem \ref{t9.4}. Thus $Ann_{R_1}(K_1)$ is a primary ideal of $R_1$ by \cite[2.23]{BUY14}. Thus by Theorem \ref{t9.5},  $K_1$ be a strongly 2-absorbing secondary submodule of $M_1$.

(b) We have similar arguments as in part (a).

(c) Let  $K_1$ be a secondary submodule of $M_1$ and $K_2$ be a secondary submodule of $M_2$. Then $Ann_{R_1}(K_1)$ and $Ann_{R_2}(K_2)$ are primary ideals of $R_1$ and $R_2$, respectively. Now since $Ann_R(N)=Ann_{R_1}(K_1)\times Ann_{R_2}(K_2)$, we have $Ann_R(N)$ is a 2-absorbing primary ideal of $R$ by \cite[2.23]{BUY14}. Thus the result follows from Theorem \ref{t9.5}.
\end{proof}

\begin{lem}\label{t9.15}
Let $N$ be a submodule of a comultiplication $R$-module $M$. Then $N$ is a secondary module if and only if $Ann_R(N)$ be a primary ideal of $R$.
\end{lem}
\begin{proof}
The necessity is clear. For converse, let $r \in R$. As $M$ is a comultiplication module, $rN=(0:_MI)$ for some ideal $I$ of $R$. Now $rI \subseteq Ann_R(N)$ implies that $I \subseteq Ann_R(N)$ or $r^t\in Ann_R(N)$ for some positive integer $t$. Thus as $M$ is a comultiplication $R$-module, $N=rN$ or $r^tN=0$ for some positive integer $t$.
\end{proof}

\begin{thm}\label{t9.16}
Let $R = R_1 \times R_2$ be a decomposable ring and $M = M_1 \times M_2$
be a finitely generated comultiplication $R$-module, where $M_1$ is an $R_1$-module and $M_2$ is an $R_2$-module. Suppose that $N = N_1 \times N_2$ is a non-zero submodule of $M$. Then the following conditions are equivalent:
\begin{itemize}
  \item [(a)] $N$ is a strongly 2-absorbing secondary submodule of $M$;
  \item [(b)] Either $N_1 = 0$ and $N_2$ is a strongly 2-absorbing secondary submodule of $M_2$ or $N_2 = 0$ and $N_1$ is a strongly 2-absorbing secondary submodule of $M_1$ or $N_1$, $N_2$ are secondary submodules of $M_1$, $M_2$, respectively.
\end{itemize}
\end{thm}
\begin{proof}
$(a) \Rightarrow (b)$.
Let $N = N_1 \times N_2$ be a strongly 2-absorbing secondary submodule
of $M$. Then $Ann_R(N)= Ann_{R_1}(N_1)\times Ann_{R_2}(N_2)$ is a 2-absorbing
primary ideal of $R$ by Theorem \ref{t9.4}. By \cite[2.23]{BUY14},
we have $Ann_{R_1}(N_1)=R_1$ and $Ann_{R_2}(N_2)$ is a 2-absorbing primary ideal of
$R_2$ or $Ann_{R_2}(N_2)=R_2$ and $Ann_{R_1}(N_1)$ is a 2-absorbing primary ideal of $R_1$  or $Ann_{R_1}(N_1)$ and $Ann_{R_2}(N_2)$ are primary ideals of $R_1$ and $R_2$, respectively. Suppose that $Ann_{R_1}(N_1)=R_1$ and $Ann_{R_2}(N_2)$ is a 2-absorbing primary ideal of $R_2$. Then $N_1=0$ and $N_2$ is a strongly 2-absorbing secondary submodule of $M_2$ by Theorem \ref{t9.5}. Similarly, if $Ann_{R_2}(N_2)=R_2$ and $Ann_{R_1}(N_1)$ is a 2-absorbing primary ideal of $R_1$, then $N_2=0$ and $N_1$ is a strongly 2-absorbing secondary submodule of $M_1$. If the last case hold, then as $M_1$ (resp. $M_2$) is a comultiplication $R_1$ (resp.  $R_2$) module, $N_1$ (resp. $N_2$) is a secondary submodule of $M_1$ (resp. $M_2$) by Lemma \ref{t9.15}.

$(b) \Rightarrow (a)$.
This follows from Theorem \ref{t9.15}.
\end{proof}

\bibliographystyle{amsplain}

\end{document}